%% file: qcat-of-frames.tex
\documentclass[11pt]{amsart}

\usepackage[utf8]{inputenc}  

\usepackage{amsthm}
\usepackage{amsfonts}
\usepackage{amsmath}
\usepackage{amssymb}
\usepackage{mathtools} 
\usepackage{stmaryrd}
\usepackage{marvosym}

\usepackage{xcolor}
\definecolor{darkgreen}{rgb}{0,0.45,0}
\definecolor{darkred}{rgb}{0.75,0,0}
\definecolor{darkblue}{rgb}{0,0,0.6}
\usepackage[colorlinks,citecolor=darkgreen,linkcolor=darkred,urlcolor=darkblue]{hyperref}
\usepackage{breakurl}
\usepackage[mathscr]{eucal}
\usepackage{enumerate} 
\usepackage[capitalize]{cleveref}

\usepackage{tikz}
  \usetikzlibrary{matrix,arrows,calc}

\usepackage{tikz}
\usetikzlibrary{arrows}

\theoremstyle{plain}
\newtheorem{theorem}{Theorem}[section]

\newtheorem{proposition}[theorem]{Proposition}
\newtheorem{lemma}[theorem]{Lemma}
\newtheorem{corollary}[theorem]{Corollary}

\theoremstyle{definition}
\newtheorem{definition}[theorem]{Definition}

\newtheorem{remark}[theorem]{Remark}

\Crefname{corollary}{Corollary}{Corollaries}
\Crefname{theorem}{Theorem}{Theorems}
\Crefname{proposition}{Proposition}{Propositions}

\usetikzlibrary{matrix}

\tikzset
{
  diagram/.style=
  {
    matrix of math nodes,
    column sep=2.5em,
    row sep=2.5em,
    text height=1.5ex,
    text depth=.25ex
  },
  every to/.style={font=\footnotesize},
}

\newcommand*\circled[1]{\tikz[baseline=(char.base)]{
            \node[shape=circle,draw,inner sep=1pt] (char) {#1};}}
            
\DeclareFontFamily{OT1}{pzcx}{}
\DeclareFontShape{OT1}{pzcx}{m}{it}{<-> s*[1.18] pzcmi7t}{}

\DeclareMathAlphabet{\mathpzc}{OT1}{pzcx}{m}{it}

\newcommand{\Cat}{\mathsf{Cat}} 
\newcommand{\hoCat}{\mathsf{hoCat}} 
\newcommand{\Kan}{\mathsf{Kan}} 
\newcommand{\qCat}{\mathsf{qCat}} 
\newcommand{\Set}{\mathsf{Set}} 
\newcommand{\sSet}{\mathsf{sSet}} 
\newcommand{\ssSet}{\mathsf{ssSet}} 
\newcommand{\weCat}{\mathsf{weCat}}

\newcommand{\C}{\mathpzc{C}} 
\newcommand{\D}{\mathpzc{D}} 
\newcommand{\M}{\mathpzc{M}} 

\newcommand{\J}{\mathrm{J}} 
\newcommand{\N}{\mathrm{N}} 
\newcommand{\Nf}{\mathrm{N}_\mathrm{f}} 

\newcommand{\colim}{\operatorname{colim}} 
\newcommand{\diag}{\mathrm{diag}} 
\newcommand{\ev}{\mathrm{ev}}  
\newcommand{\Ho}{\mathrm{Ho}} 
\newcommand{\Lan}{\mathrm{Lan}} 
\newcommand{\ob}{\operatorname{Ob}} 
\newcommand{\op}{\mathrm{op}} 
\newcommand{\Sd}{\mathrm{Sd}} 

\newcommand{\ha}[1]{\widehat{#1}}

\newcommand{\bbN}{\mathbb{N}} 
\newcommand{\bbR}{\mathbb{R}} 
\newcommand{\bbZ}{\mathbb{Z}} 

\newcommand{\adjoint}{\dashv} 

\renewcommand{\to}{\rightarrow} 
\newcommand{\into}{\hookrightarrow} 
\newcommand{\iso}{\cong} 

\newcommand{\R}{\mathrm{R}}
\newcommand{\Nsub}[1]{\mathrm{N}_{#1}}
\newcommand{\Ncd}{\mathbf{N}}
\newcommand{\bfNf}{\mathbf{N}_{\mathrm{f}}}

\newcommand{\ucat}[1]{{#1}^{\operatorname{cat}}}
\newcommand{\usp}[1]{{#1}^{\operatorname{sp}}}

\newcommand{\from}{\colon}               
\newcommand{\ito}{\hookrightarrow}       
\newcommand{\weto}{\stackrel{\we}{\to}}  

\newcommand{\cat}[1]{\mathpzc{#1}}    
\newcommand{\qcat}[1]{\mathscr{#1}}   

\newcommand{\nf}{\operatorname{N_f}}  

\newcommand{\cof}{\mathrm{cof}}
\newcommand{\fibr}{\mathrm{fib}}

\newcommand{\bd}{\mathord\partial} 

\newcommand{\slice}{\downarrow}

\newcommand{\we}{\sim}    

\renewcommand{\hat}{\widehat}         
\renewcommand{\tilde}{\widetilde}

\newcommand{\simp}[1]{\Delta[#1]}       
\newcommand{\bdsimp}[1]{\bd \Delta[#1]} 

\newcommand{\w}{\mathrm{w}}
\newcommand{\Reedy}{\mathrm{R}}

\newcommand{\Ex}{\operatorname{Ex}}
\newcommand{\EX}{\operatorname{\mathbf{Ex}}}

\newcommand{\id}{\operatorname{id}}

\makeatletter
\def\horn#1{\expandafter\horn@i#1,,\@nil}
\def\horn@i#1,#2,#3\@nil{\Lambda^{#2}[#1]}
\makeatother

\newcommand{\st}{such that}

\renewcommand{\phi}{\varphi}

\newcommand{\mD}{\mathbb{D}} 

\newcommand{\mnf}{\mathbb{N}_{\operatorname{f}}}  

\input{tikz-diagrams}

\begin{document}
\title{Quasicategories of frames of cofibration categories}
\author{Krzysztof Kapulkin \and Karol Szumi{\l}o}
\date{\today}
\begin{abstract}
  We show that the quasicategory of frames of a cofibration category, introduced by the second-named author,
  is equivalent to its simplicial localization.
\end{abstract}
\maketitle
\input{0-introduction}

\input{1-hoalgebra}

\input{2-fibcat}

\input{3-diagrams}

\input{4-comp-with-classification}

\input{5-comp-in-model}

\bibliography{general-bibliography}

\end{document}

%% file: tikz-diagrams.tex
\tikzset
{
  diagram/.style=
  {
    matrix of math nodes,
    column sep=2.5em,
    row sep=2.5em,
    text height=1.5ex,
    text depth=.25ex
  },
  cross line/.style={preaction={draw=white,-,line width=6pt}},
  every to/.style={font=\footnotesize},
  cof/.style={>->},                
  fib/.style={->>},                
  inj/.style={right hook->},       
  surj/.style={-open triangle 60}, 
  eq/.style={-,double distance=1.7pt}
}

\makeatletter
\newenvironment{ctikzpicture}
{
  \begingroup
  \smallskip
  \par
  \centering
  \begin{tikzpicture}
}{
  \end{tikzpicture}
  \par
  \smallskip
  \endgroup
  \aftergroup\@afterindentfalse
  \aftergroup\@afterheading
}
\makeatother

%% file: 0-introduction.tex
\section*{Introduction} \label{sec:intro}

Starting with the work of Gabriel and Zisman \cite{gabriel-zisman}, categories with weak equivalences have been used to study homotopy theories. Later, thanks to the results of Dwyer and Kan \cite{dwyer-kan:simplicial-localizations, dwyer-kan:calculate-localizations, dwyer-kan:complexes}, it became clear that the content of a homotopy theory is entirely captured by the notion a category with weak equivalences and a precise formulation of this observation was eventually given by Barwick and Kan \cite{barwick-kan}.

More precisely, they showed that the homotopy theory of categories with weak equivalences is equivalent to the homotopy theory of $(\infty,1)$-categories (presented as quasicategories or complete Segal spaces). The latter are often more convenient in practice and hence it is important to understand simplicial localization functors, i.e.\ functors associating to a category with weak equivalences the corresponding higher category. (Examples of such constructions include the classification diagram of Rezk \cite{rezk:css} and the hammock localization \cite{dwyer-kan:calculate-localizations} followed by the derived homotopy coherent nerve.)

A common problem arising while working with these constructions is the necessity of using inexplicit fibrant replacements.
These problems can be avoided if the category with weak equivalences is known to possess more structure, namely, when it is a cofibration category (or a fibration category). Indeed, given a cofibration category $\C$, one can associate to it its quasicategory of frames $\Nf\C$, introduced by the second-named author \cite{szumilo:two-models}.

The main goal of this paper is a proof that the quasicategory of frames and other constructions of simplicial localization are equivalent. Specifically, we define an enhancement of the quasicategory of frames to a complete Segal space and show that it is equivalent to the classification diagram. From this, using the results of To\"en \cite{toen:unicity}, we deduce equivalence with other notions.

In the upcoming work of the first-named author \cite{kapulkin:lccqcat-from-type-theory} our results will be used to show that the simplicial localization of any categorical model of Homotopy Type Theory is necessarily a locally cartesian closed quasicategory. Every categorical model of type theory is known to carry the structure of a fibration category \cite{avigad-kapulkin-lumsdaine} and, by our results, its simplicial localization can be realized as the quasicategory of frames. This realization proved convenient for the purpose of solving the problem in question.

The paper is organized as follows. In \Cref{sec:hoalgebra}, we review the relevant background on models of homotopy theories (or, equivalently, $(\infty,1)$-categories). In \Cref{sec:fibcat}, we collect the necessary facts about cofibration categories and the construction of the quasicategory of frames. \Cref{sec:diagrams} contains the technical heart of the paper---a proof of the compatibility of $\Nf$ with formation of diagrams, which is then used in \Cref{sec:comp-with-classification} to establish our main theorem relating the quasicategory of frames to the classification diagram. In particular, it follows that given a model category, the quasicategories of frames associated to its underlying cofibration and fibration categories are equivalent. In \Cref{sec:comp-in-model}, we supply a more direct comparison of these quasicategories.

%% file: 1-hoalgebra.tex
\section{Models of homotopy theories}\label{sec:hoalgebra}

In this section, we present three models of the homotopy theory of homotopy theories: categories with weak equivalences, quasicategories, and complete Segal spaces. For future reference, we will also recall some of their basic properties.

A \textbf{category with weak equivalences} consists of a category $\C$ together with a wide subcategory $\w\C$ (i.e.\ a subcategory containing all objects of $\C$).
Morphisms of $\w\C$ will be referred to as \textbf{weak equivalences}.
A functor $F \colon \C \to \D$ between categories with weak equivalences is
\textbf{homotopical} if it takes weak equivalences of $\C$ to weak equivalences
of $\D$.

A homotopical functor $F \colon \C \to \D$ is a \textbf{Dwyer--Kan
equivalence} (or \textbf{DK-equivalence} for short) if it induces an
equivalence $\Ho F$ of homotopy categories and a weak homotopy equivalence on
mapping spaces in the hammock localizations of $\C$ and $\D$ (see \cite{dwyer-kan:simplicial-localizations, dwyer-kan:calculate-localizations}). This notion
naturally implements the idea of equivalence of homotopy theories---two
homotopy theories (presented as categories with weak equivalences) are
considered the same if their homotopy categories and mapping spaces agree. 

We will write $\weCat$ for the category of small categories with weak equivalences and consider it as a category with weak equivalences with Dwyer--Kan equivalences as weak equivalences.

A \textbf{quasicategory} is a simplicial set $\mathscr{C}$ satisfying the inner horn filling condition, i.e.\ for every $0 < i < m$ and every $\horn{m,i} \to \mathscr{C}$, there exists a filler:
\begin{ctikzpicture}
  \matrix[diagram]
  {
    |(h)| \horn{m,i} & |(C)| \qcat{C} \\
    |(s)| \simp{m} \\
  };

  \draw[->]  (h) to (C);
  \draw[inj] (h) to (s);

  \draw[->,dashed] (s) to (C);
\end{ctikzpicture}

We will write $\qCat$ for the full subcategory of $\sSet$ whose objects are quasicategories.

Given a category $\C$, one associates to it a quasicategory $\N\C$, called the \textbf{nerve} of $\C$, whose $m$-simplices are given by functors $[m] \to \C$. We will write $E[1]$ for the nerve of a contractible groupoid with two objects $0$ and $1$.

The category $\sSet$ can also be equipped with a class of maps, called 
categorical equivalences, playing the role of equivalences of homotopy theories.
We first need introduce the notion of an $E[1]$-homotopy. Two maps $f, g \colon K \to L$ of simplicial sets are $E[1]$-\textbf{homotopic} if there exists a map $H \colon K \times E[1] \to L$ whose restriction to $K \times \bdsimp{1}$ is $[f,g]$. A map $w \colon K \to L$ is a \textbf{categorical equivalence} if the induced map $[L, \mathscr{C}]_{E[1]} \to [K, \mathscr{C}]_{E[1]}$ is a bijection for every quasicategory $\mathscr{C}$, where $[X, Y]_{E[1]}$ denotes the set of $E[1]$-homotopy classes of maps $X \to Y$.

Another class of examples of quasicategories is given by \textbf{Kan complexes}, which satisfy a stronger version of the horn filling condition; that is, they are required to have horn fillers for all horns (i.e.\ we take $0 \leq i \leq m$). The full subcategory of $\qCat$ whose objects are Kan complexes will be denoted $\Kan$.
The inclusion $\Kan \into \qCat$ admits a right adjoint $\J \colon \qCat \to
\Kan$ picking out the largest Kan
complex contained in a quasicategory \cite[Thm.\ 4.19]{joyal:theory-of-quasi-cats}.

\begin{proposition}[{\cite[Prop.\ 4.26]{joyal:theory-of-quasi-cats}}]\label{prop:J-preserves-equiv}
 $\J$ carries categorical equivalences of quasicategories to homotopy equivalences of Kan complexes.\qed
\end{proposition}

Lastly, we will need the notion of an inner isofibration. Recall that a map is an \textbf{inner fibration} if it has the right lifting property with respect to all inner horn inclusions, i.e.\ $\horn{m,i} \into \Delta[m]$ for $0 < i < m$. An \textbf{inner isofibration} is a map that that is an inner fibration and, in addition, has the right lifting property with respect to the inclusion $\delta_1 \colon \Delta[0] \into E[1]$.

As our last model for the homotopy theory of homotopy theories, we shall discuss complete Segal spaces. Before doing that, let us introduce some notation. Given a bisimplicial set $W \colon \Delta^\op \times \Delta^\op \to \Set$, we may regard it as a simplicial object $W \colon \Delta^\op \to \sSet$ in two different ways. This gives us two different Kan extensions of $W$ along the Yoneda embedding that we will denote $\usp{W}$ and $\ucat{W}$, respectively. We will also write $\usp{W}_m$ for $\usp{W}(\Delta[m])$ and $\ucat{W}_n$ for $\ucat{W}(\Delta[n])$.

A bisimplicial set $W$ is a \textbf{complete Segal space} if it satisfies the following conditions:
\begin{enumerate}
 \item it is Reedy fibrant, i.e.\ the canonical map $\usp{W}_m \to \usp{W}(\partial
\Delta [m])$ is a Kan fibration for all $m \in \bbN$;
 \item it is a Segal space, i.e.\ the canonical map $\usp{W}_m \to \usp{W}(S[m])$ is a
weak homotopy equivalence for all $m \in \bbN$, where $S[m]$ is the simplicial
subset of $\Delta[m]$ consisting of all vertices and edges connecting all pairs
of consecutive vertices (the \emph{spine} of $\Delta[m]$);
 \item it is complete, i.e.\ the canonical map $\usp{W}_0 \to \usp{W}(E[1])$ is a weak
homotopy equivalence.
\end{enumerate}

A map of bisimplicial sets $w \colon X \to Y$ is a \textbf{Rezk equivalence} if
for every complete Segal space $W$ the induced map $W^Y \to W^X$ is a levelwise
weak homotopy equivalence. In particular, every levelwise weak homotopy
equivalence of bisimplicial sets is a Rezk equivalence.

\begin{proposition}[{\cite[Prop.\ 4.4]{joyal-tierney:qcat-vs-segal}}] \label{prop:css-is-frame-in-qcat}
 A bisimplicial set $W$ is a complete Segal space if it is a frame in the
category $\qCat$, i.e.\
\begin{enumerate}
 \item it is Reedy fibrant (the canonical map $\ucat{W}_n \to \ucat{W}(\partial
\Delta[n])$ is an inner isofibration for all $n \in \bbN$);
 \item it is homotopically constant (every simplicial operator $[n] \to [n']$
induces a categorical equivalence $\ucat{W}_{n'} \to \ucat{W}_n$).\qed
\end{enumerate}
\end{proposition}

\begin{lemma} \label{lem:rezk-is-level-qcat}
 A Rezk equivalence $w \colon X \to Y$ between complete Segal spaces is a levelwise categorical equivalence (i.e.\ $\ucat{w}_n \colon \ucat{X}_n \to \ucat{W}_n$ is a categorical equivalence of quasicategories for all $n \in \bbN$).
\end{lemma}

\begin{proof}
 See the proof of \cite[Prop.\ 4.7]{joyal-tierney:qcat-vs-segal}.
\end{proof}

Let $\C$ be a category with weak equivalences. The \textbf{classification diagram} of $\C$ (cf. \cite[Sec.\ 3.3]{rezk:css}) is a bisimplicial set $\Ncd\C$ whose $(m,n)$-simplices are given by:
\[ (\Ncd\C)_{m,n} = \left\{\text{homotopical functors } [m] \times \hat{[n]} \to \C \right\}\text{.} \]
Here, in $[m]$ we take only identity maps as weak equivalences, while in $\hat{[n]}$ all maps are weak equivalences. Alternatively, one may describe $\Ncd\C$ by: $\usp{(\Ncd\C)}_m = \N\w(\C^{[m]})$, where the weak equivalences in the category $\C^{[m]}$ are the natural weak equivalences (i.e.\ natural transformations whose components are weak equivalences). The functor $\Ncd \colon \weCat \to \ssSet$ is a DK-equivalence by \cite[Lem.\ 5.4, Thm.\ 6.1(i), Prop.\ 10.3]{barwick-kan}.

%% file: 2-fibcat.tex
\section{Cofibration categories and the quasicategory of frames} \label{sec:fibcat}

In this section, we will review the background on cofibration categories and, as indicated in the Introduction, will take advantage of the structure of a cofibration category to produce a convenient model for its simplicial localization, called the quasicategory of frames. This construction was introduced in \cite{szumilo:two-models}; here, we summarize the relevant notions and techniques of this paper.

\begin{definition}
 A \textbf{cofibration category} consists of a category $\C$ together with two wide subcategories: of \textbf{cofibrations} and of \textbf{weak equivalences} such that (in what follows, an \textbf{acyclic fibration} is a morphism that is both a cofibration and a weak equivalence):
 \begin{enumerate}
  \item the class of weak equivalences satisfies \textbf{2-out-of-6} property; that is, given a composable triple of morphisms:
  \[ X \overset{f}\longrightarrow Y \overset{g}\longrightarrow Z \overset{h}\longrightarrow Z \]
  if $hg, gf$ are weak equivalences, then so are $f$, $g$, and $h$.
  \item all isomorphisms are acyclic cofibrations.
  \item pushouts along cofibrations exist; cofibrations and acyclic cofibrations are stable under pushouts.
  \item $\C$ has an initial object $0$; the canonical morphism $0 \to X$ is a cofibration for any object $X \in \C$ (that is, all objects are \textbf{cofibrant}).
  \item every morphism can be factored as a cofibration followed by a weak equivalence.
 \end{enumerate}
\end{definition}

Given a model category, its subcategory of cofibrant objects is a cofibration category. There are, however, plenty of examples of cofibration categories that do not arise as the subcategory of cofibrant objects in a model category, e.g.\  the category of topological spaces and proper maps (see \cite[Sec.\ 1.4]{szumilo:two-models} for a discussion of such examples).

There is also the dual notion of a \textbf{fibration category}. A fibration category consists of a category $\C$, together with two classes of maps: fibrations and weak equivalences, subject to the axioms dual to these of a cofibration category. The category $\qCat$ of quasicategories carries a structure of a fibration category, in which weak equivalences are categorical equivalences and fibrations are inner isofibrations. This category arises as the subcategory of fibrant objects in Joyal's model structure on simplicial sets.

\begin{definition}\leavevmode
\begin{enumerate}
 \item A functor between cofibration categories is \textbf{exact} if it preserves cofibrations, acyclic cofibrations, pushouts along cofibrations, and an initial object.
 \item An exact functor is a \textbf{weak equivalence} of cofibration categories if it induces an equivalence of homotopy categories.
\end{enumerate}
\end{definition}

(Again, there is a dual notion of an exact functor between fibration categories; such a functor is required to preserve fibrations, acyclic fibrations, pullbacks along fibrations, and a terminal object.)

The following theorem gives a useful characterization of weak equivalences between cofibration categories:

\begin{theorem}[{\cite[Thm.\ 3.19]{cisinski:categories}}] \label{thm:approximation-properties}
 An exact functor $F \colon \C \to \D$ between cofibration categories is a weak equivalence if and only if it satisfies the following \textbf{Approximation Properties}:
 \begin{enumerate}
  \item[(App1)] $F$ reflects weak equivalences;
  \item[(App2)] given a morphism $f \colon FA \to Y$ in $\D$, there exists a morphism $i \colon A \to B$ in $\C$ and a commutative square:
    \begin{ctikzpicture}
      \matrix[diagram]
      {
        |(A)| FA & |(Y)| Y \\
        |(B)| FB & |(Z)| Z \\
      };

      \draw[->] (A) to node[above] {$f$}  (Y);
      \draw[->] (A) to node[left]  {$Fi$} (B);

      \draw[->] (Y) to node[right] {$\we$} (Z);
      \draw[->] (B) to node[above] {$\we$} (Z);
    \end{ctikzpicture}
  in $\D$. \qed
 \end{enumerate}
\end{theorem}

One can also define the notion of a fibration between between cofibration categories. An exact functor $P \colon \C \to \D$ is a \textbf{fibration} if it satisfies the following conditions:
 \begin{enumerate}
  \item $P$ is an isofibration;
  \item given a map $f \colon A \to B$ in $\C$ and a factorization $Pf = t j$ of $Pf$ as a cofibration followed by a weak equivalence, there exists a factorization $f = s i$ of $f$ into a cofibration followed by a weak equivalence such that $Pi =j$ and $Ps = t$.
  \item given a map $f \colon A \to B$ in $\C$ and a commutative square:
    \begin{ctikzpicture}
      \matrix[diagram]
      {
        |(A)| PA & |(B)| PB \\
        |(X)| X  & |(Y)| Y \\
      };

      \draw[->] (A) to node[above] {$Pf$}  (B);
      \draw[->] (X) to node[below] {$t$} node[above] {$\we$} (Y);

      \draw[->,cof] (A) to node[left]  {$j$} (X);
      \draw[->,cof] (B) to node[right] {$v$} node[left] {$\we$} (Y);
    \end{ctikzpicture}
  in $\D$, in which $j$ is a cofibration, $t$ is a weak equivalence, and $v$ is an acyclic cofibration, there is a commutative square:
    \begin{ctikzpicture}
      \matrix[diagram]
      {
        |(A)| A & |(B)| B \\
        |(C)| C & |(D)| D \\
      };

      \draw[->] (A) to node[above] {$f$} (B); \draw[->] (C) to
      node[below] {$s$} node[above] {$\we$} (D);

      \draw[->,cof] (A) to node[left]  {$i$} (C);
      \draw[->,cof] (B) to node[right] {$u$} node[left] {$\we$} (D);
    \end{ctikzpicture}
  in $\C$, in which $i$ is a cofibration, $s$ is a weak equivalence, and $u$ is an acyclic cofibration such that $Pi = j$, $Ps = t$, and $Pu =v$.
 \end{enumerate}

\begin{theorem}[{\cite[Thm.\ 1.14]{szumilo:two-models}}] \label{thm:fib-cat-of-cof-cat}
 The category of cofibration categories and exact functors with fibrations and weak equivalences defined above is a fibration category. \qed
\end{theorem}

The definition of the quasicategory of frames (and its enhancement to a complete Segal space) will depend on the notion of a Reedy cofibrant diagram on a direct category. We therefore review the necessary definitions.

\begin{definition}\leavevmode
\begin{enumerate}
 \item A category $J$ is \textbf{direct} if there is a function, called \textbf{degree}, $\deg \colon \ob(J) \to \bbN$ such that for every non-identity map $j \to j'$ in $J$ we have $\deg(j) > \deg(j')$.
\end{enumerate}
Let $J$ be a direct category.
\begin{enumerate}
 \item[(2)] Let $j \in J$. The \textbf{latching category} $\partial(J \downarrow j)$ of $j$ is the full subcategory of the slice category $J \downarrow j$ consisting of all objects except $\id_j$. There is a canonical functor $\partial(J \downarrow j) \to J$, assigning to a morphism (regarded as an object of $\partial(J \downarrow j)$) its domain.
 \item[(3)] Let $X \colon J \to \C$ and $j \in J$. The \textbf{latching object} of $X$ at $j$ is defined as a colimit of the composite
   \[ L_j X := \colim (\partial(J \downarrow j) \longrightarrow J \overset{X}{\longrightarrow} \C).\]
   The canonical morphism $L_j X \to X_j$ is called the \textbf{latching morphism}.
 \item[(4)] Let $\C$ be a cofibration category. A diagram $X \colon J \to \C$ is called \textbf{Reedy cofibrant}, if for all $j \in J$, the latching object $L_j X$ exists and the latching morphism $L_j X \to X_j$ is a cofibration.
 \item[(5)] Let $\C$ be a cofibration category and let $X, Y \colon J \to \C$ be Reedy cofibrant diagrams in $\C$. A morphism $f \colon X \to Y$ of diagrams is a \textbf{Reedy cofibration} if for all $j \in J$ the induced morphism $ X_j \sqcup_{L_j X} L_j Y \to Y_j$ is a cofibration.
\end{enumerate}
\end{definition}

Recall that a \textbf{homotopical category} is a category with weak equivalences satisfying the 2-out-of-6 property. We will denote by $\hoCat$ the full subcategory of $\weCat$ whose objects are homotopical categories. We will restrict our attention to homotopical categories, because the techniques of \cite{szumilo:two-models} are well-adapted for this notion. Given a small homotopical category $J$, we will construct a direct homotopical category $DJ$ (a ``direct approximation'' of $J$), together with a homotopical functor $p \colon DJ \to J$. The objects of $DJ$ are pairs $([m], \varphi \colon [m] \to J)$ with varying $n \in \bbN$. A morphism
  \[ f \colon ([m], \varphi) \to ([n], \psi)  \]
 is an injective, order preserving map $f \colon [m] \into [n]$ making the following triangle commute:
 \begin{ctikzpicture}
   \matrix[diagram]
   {
     |(m)| [m] && |(n)| [n] \\
     & |(J)| J \\
   };

   \draw[->,inj] (m) to node[above] {$f$} (n);

   \draw[->] (m) to node[below left]  {$\phi$} (J);
   \draw[->] (n) to node[below right] {$\psi$} (J);
 \end{ctikzpicture}
 It is clear that $DJ$ is a direct category (with $\deg ([m], \varphi) = m$). To define $p \colon DJ \to J$ we put $p([m],\varphi)=\varphi(m)$. Finally, we declare that a map $w$ in $DJ$ is a weak equivalence if $p(w)$ is a weak equivalence in $J$. This makes $DJ$ into a category with weak equivalences and $p$ into a homotopical functor. 
 
 \begin{definition}
 Let $\C$ be a cofibration category. We define the simplicial set $\Nf\C$, called the \textbf{quasicategory of frames} in $\C$, by setting:
 \[ (\Nf\C)_m :=\left\{ \text{homotopical, Reedy cofibrant diagrams } D[m] \to \C \right\} . \]
\end{definition} 

\begin{theorem}[{\cite[Thm.\ 3.3]{szumilo:two-models}}]
 For any cofibration category $\C$, the simplicial set $\Nf\C$ is a quasicategory and moreover, $\Nf$ is an exact functor from the fibration category of cofibration categories (of \Cref{thm:fib-cat-of-cof-cat}) to the fibration category of quasicategories. \qed
\end{theorem}

In fact, more is true: for a cofibration category $\C$, the quasicategory $\Nf\C$ can be shown to possess all finite colimits. Moreover, $\Nf$ is a weak equivalence between the fibration category of cofibration categories and the fibration category of finitely cocomplete quasicategories \cite[Thm.\ 2.17 and 4.11]{szumilo:two-models}. Let us also record that by Ken Brown's Lemma, we obtain the following corollary:

\begin{corollary}\label{cor:Nf-preserves-equiv}
 $\Nf$ carries weak equivalences of cofibration categories to categorical equivalences of quasicategories. \qed
\end{corollary}

One of the goals of the present work is to establish an equivalence between $\Nf$ and other constructions of simplicial localization. For this purpose we introduce the following enhancement of the quasicategory of frames to a complete Segal space.

\begin{definition}\label{def:bold-Nf}
 Given a cofibration category $\C$, we define a bisimplicial set $\bfNf\C$ by:
 \[ (\bfNf\C)_{m,n} :=\left\{\text{homotopical, Reedy cofibrant diagrams } D([m] \times \hat{[n]}) \to \C  \right\}. \]
\end{definition}

\begin{remark}\label{rmk:jt-qcat-to-css}
This definition is inspired by the construction of Joyal and Tierney, assigning to a quasicategory $\mathscr{C}$, a complete Segal space $\J(\mathscr{C}^{\Delta[-]})$ \cite[p.\ 24]{joyal-tierney:qcat-vs-segal}. Unwinding the definitions, one can check that $\bfNf\C$ is given by applying their construction to $\Nf\C$. It follows that $\bfNf\C$ is a complete Segal space for any cofibration category $\C$. 
\end{remark}

Our main result (\Cref{thm:Nf-and-Ncd}) shows that the bisimplicial sets $\Ncd\C$ and $\bfNf\C$ are Rezk equivalent. (We also point out that putting $n=0$, i.e.\ taking the $0$th row, yields $\ucat{(\bfNf\C)}_0 \cong \Nf\C$.) 

In the remainder of this section, we will collect several lemmas needed in the subsequent sections.
We begin, however, with two auxiliary constructions.

Given a poset $P$, define a direct category $\Sd P$ with weak equivalences as the full subcategory of $DP$ whose objects are injective monotone functions $\varphi \colon [n] \into P$, i.e.\ non-empty chains in $P$. The weak equivalences of $\Sd P$ are created by the functor $\max \colon \Sd P \to P$, taking a chain to its maximal element, or, equivalently, by the inclusion $\Sd P \into DP$ (notice that $\max$ is simply the restriction of $p \colon DP \to P$ to the subcategory $\Sd P$).

 Similarly, we may define $D$ for simplicial sets, rather than for categories. Let $K \in \sSet$ and define the underlying category of $DK$ to be the category of elements of $K$, considered as a semisimplicial set (i.e.\ without degeneracy maps). The set of weak equivalences in $DK$ are the smallest set closed under 2-out-of-6 and containing the morphisms induced by the degenerate $1$-simplices of $K$.
 This definition can be extended to simplicial sets with certain extra structure, but we will only need one instance of that, so we will give an ad hoc definition. Namely, let $D\hat{\partial \Delta[n]}$ denote the homotopical category with $D(\partial \Delta[n])$ as its underlying category and all maps as weak equivalences.
 
\begin{proposition}[{\cite[Prop.\ 3.7]{szumilo:two-models}}]
 Let $\C$ be a cofibration category and $K$ a simplicial set. There is a natural bijection between the set of simplicial maps $K \to \Nf\C$ and the set of Reedy cofibrant diagrams $DK \to \C$. \qed
\end{proposition}

The remaining lemmas will establish several properties of the cofibration categories of diagrams.

\begin{proposition}\label{prop:reedy-vs-levelwise}
 Let $\C$ be a cofibration category and $J$ a direct category with weak equivalences and finite latching categories.
 \begin{enumerate}
  \item The category $\C^J_\R$ of homotopical, Reedy cofibrant diagrams $J \to \C$ is a cofibration category, in which: weak equivalences are levelwise weak equivalences and cofibrations are Reedy cofibrations \cite[Thm.\ 9.3.8(1a)]{radulescu-banu}.
  \item The category $\C^J$ of all homotopical diagrams $J \to \C$ is a cofibration category, in which: weak equivalences are levelwise weak equivalences and cofibrations are levelwise cofibrations \cite[Thm.\ 9.3.8(1b)]{radulescu-banu}.
  \item The canonical inclusion $\C^J_\R \into \C^J$ is a weak equivalence of cofibration categories \cite[Prop.\ 1.16(3)]{szumilo:two-models}. \qed
 \end{enumerate}
\end{proposition}

\begin{lemma}[{\cite[Lem.\ 3.9]{szumilo:two-models}}] \label{lem:Dm-to-m}
 The map $p \colon D[m] \to [m]$ is a homotopy equivalence and thus induces weak equivalences of cofibration categories of diagrams (both for Reedy and levelwise structures). \qed
\end{lemma}

\begin{lemma}\label{lem:curry-reedy}
 For a cofibration category $\C$ and direct categories $I$ and $J$, the cofibration categories of diagrams $\C^{I \times J}_\R$ and $(\C^I_\R)^J_\R$ are equivalent.
\end{lemma}

\begin{proof}
 The latching categories satisfy the Leibniz formula \cite[Ex.\ 4.6]{riehl-verity:theory-and-practice} and thus a morphism of $\C^{I \times J}_\R$ is a cofibration if and only if the corresponding morphism of $(\C^I_\R)^J_\R$ is.
\end{proof}

Recall that a map $I \to J$ of small categories is a \textbf{sieve} if it is injective on objects, fully faithful, and if $j \to i$ is a morphism of $J$ such that $i \in I$, then $j \in I$. 

\begin{lemma}\label{lem:reedy-lifting}
 Every acyclic fibration $P \colon \C \to \D$ of cofibration categories has the Reedy right lifting property with respect to sieves between direct categories with weak equivalences with finite latching objects, i.e.\ every square of the form:
 \begin{ctikzpicture}
   \matrix[diagram]
   {
     |(I)| I & |(C)| \cat{C} \\
     |(J)| J & |(D)| \cat{D} \\
   };

   \draw[->] (I) to (C);
   \draw[->] (J) to (D);

   \draw[->,inj] (I) to (J);
   \draw[->,fib] (C) to node[right] {$P$} (D);
 \end{ctikzpicture}
in which  $I \to J$ is a sieve of direct categories with weak equivalences with finite latching categories and the horizontal arrows are Reedy cofibrant, admits a diagonal filler $J \to \C$, which is Reedy cofibrant.  
\end{lemma}

\begin{proof}
 Implication (1) $\Rightarrow$ (2) in \cite[Lem.\ 1.24]{szumilo:two-models}.
\end{proof}

Let us point out that not every functor $f \colon I \to J$ between direct categories induces an exact functor between the corresponding categories of Reedy cofibrant diagrams. The following lemma gives a useful criterion for checking the exactness.

\begin{lemma} \label{lem:exactness-criterion}
 Let $f \colon I \to J$ be a functor between direct categories such that for each $i \in I$, the canonical map $\partial (I \downarrow i) \to \partial (J \downarrow f(i))$ factors as the composite of a cofinal functor followed by
a sieve
\begin{ctikzpicture}
  \matrix[diagram]
  {
    |(I)| \bd(I \slice i) & |(K)| K & |(J)| \bd(J \slice f(i)) \text{.} \\
  };

  \draw[->] (I) to (K);
  \draw[->] (K) to (J);
\end{ctikzpicture}
 Then, for any cofibration category $\C$, the induced functor $f^* \colon \C^J_\R \to \C^I_\R$ is exact.
\end{lemma}

\begin{proof}
 Consider a Reedy cofibrant diagram $X \in \C^J_\R$ and $i \in I$. We need to show that the latching map $L_i f^*X \to (f^*X)_i$ is a cofibration. It factors as:
 \[ L_i f^*X = \colim_{\partial(I \downarrow i)} f^*X \to \colim_{K} X \to \colim_{\partial(J \downarrow f(i))} X \to X_{f(i)} = (f^*X)_i \]
 The first of these arrows is an isomorphism by the cofinality assumption; the second is a cofibration, by \cite[Thm.\ 9.4.1.(1a)]{radulescu-banu}; and the third is a cofibration since $X$ was assumed to be Reedy cofibrant.
 
 A similar argument shows that $f^*$ preserves cofibrations.
\end{proof}

The remaining two lemmas contain technical results on diagrams in cofibration categories.

\begin{lemma}\label{lem:equiv-cofinal}
  Let $f \from I \to J$ be a homotopical functor
  between finite homotopical direct categories and $\cat{C}$ a cofibration category.
  If $f$ induces a weak equivalence $\cat{C}^J_\Reedy \to \cat{C}^I_\Reedy$,
  then for every homotopical Reedy cofibrant diagram $X \from J \to \cat{C}$
  the induced morphism $\colim_I f^* X \to \colim_J X$ is a weak equivalence.
\end{lemma}

\begin{proof}
  The left Kan extension functor
  $\Lan_f \from \cat{C}^I_\Reedy \to \cat{C}^J_\Reedy$ exists,
  is exact by \cite[Thm.\ 9.4.3(1)]{radulescu-banu} and is a left adjoint of $f^*$.
  Hence $\Lan_f$ is a weak equivalence since $f^*$ is.
  In particular, the counit $\Lan_f f^* X \to X$ is a weak equivalence
  and hence so is the resulting morphism $\colim_J \Lan_f f^* X \to \colim_J X$
  which coincides with the morphism $\colim_I f^* X \to \colim_J X$.
\end{proof}

\begin{lemma}[{\cite[Lem.\ 1.19(i)]{szumilo:two-models}}] \label{lem:extend-reedy}
 Let $I \into J$ be a sieve and let $X \colon J \to \C$ be a diagram whose restriction $X|I$ is Reedy cofibrant. Then there exists a Reedy cofibrant diagram $\tilde{X} \colon J \to \C$ together with a weak equivalence $\tilde{X} \to X$ whose restriction to $I$ is the identity map (thus, in particular, we have $\tilde{X}|I = X|I$). \qed
\end{lemma}

%% file: 3-diagrams.tex
\section{Compatibility with categories of diagrams}\label{sec:diagrams}

The goal of this section is to show that for any cofibration category $\C$ and any $k \in \bbN$, the quasicategories $\Nf(\C^{D[k]}_\R)$ and $(\Nf\C)^{\Delta[k]}$ are equivalent (\Cref{thm:compatibility-Dm-Delta-m}). We will introduce a technical notion of an adequate cosimplicial object (\Cref{def:protoframe}), which abstracts the properties of the functor $D$ that ensure that $\Nf\C$ is a quasicategory for any cofibration category $\C$. Indeed, every adequate cosimplicial object yields a functor from the category of cofibration categories to the category of quasicategories (\Cref{prop:proto-to-qcat}) and also to the category of complete Segal spaces (\Cref{thm:JNf-CSS}). We point out that the latter is different than the former followed by the construction of \Cref{rmk:jt-qcat-to-css} and in fact, the key step in the proof is a comparison between the two in a relevant special case. 

We begin, however, by introducing two adequate cosimplicial objects $D[k] \times D[-]$ and $D([k] \times [-])$; and verifying that they are equivalent in the sense of \Cref{prop:DmxDn-Dmxn}.

\begin{definition} \label{def:protoframe}
 A cosimplicial object $A \colon \Delta \to \hoCat$ is \textbf{adequate} if:
 \begin{enumerate}
  \item $A[m]$ is direct for all $[m] \in \Delta$, and for every cofibration category $\C$ and every simplicial operator $[m] \to [m']$, the induced functor $\C^{A[m']}_\R \to \C^{A[m]}_\R$ is exact;
  \item the latching map $\partial A[m] \into A[m]$ is a sieve for all $[m] \in \Delta$;
  \item for all cofibration categories $\C$ and all natural numbers $0 < i < m$, the functor $\C^{A[m]}_\R \to \C^{A(\horn{m,i})}_\R$ is an equivalence of cofibration 
categories (here, we write $A(\horn{m,i})$ for the left Kan extension of $A$ along the Yoneda embedding);
  \item for all cofibration categories $\C$, the map $\C^{A(E[1])}_\R \to \C^{A[0]}_\R$ is an equivalence of cofibration categories.
 \end{enumerate}
\end{definition}

\begin{lemma}
 The cosimplicial object $D \colon \Delta \to \hoCat$ is adequate.
\end{lemma}

\begin{proof}
Condition (1) follows from \cite[Lem.\ 3.1]{szumilo:two-models}. By the proof of \cite[Prop.\ 3.7]{szumilo:two-models}, $DK$ as defined in \Cref{sec:fibcat} is the left Kan extension of $D \colon \Delta \to \hoCat$ along the Yoneda embedding. Thus (2) follows, (3) follows by the proof of \cite[Prop.\ 3.12]{szumilo:two-models}, and (4) follows by \cite[Lem.\ 3.13]{szumilo:two-models}.
\end{proof}

\begin{lemma}\label{lem:DkxD-protoframe}
 For any $k \in \bbN$, the cosimplicial object $D[k] \times D[-] \colon \Delta \to \hoCat$ is a adequate.
\end{lemma}

\begin{proof}
 Direct categories and sieves are stable under products and thus condition (2) follows. For (1) we also use \Cref{lem:curry-reedy}. Finally, for (3) and (4), we use \Cref{lem:curry-reedy} again to reduce it to the case of $D$.
\end{proof}

\begin{lemma}\label{lem:equiv-of-protoframes}
 Suppose $A, B \colon \Delta \to \hoCat$ satisfy conditions (1) and (2) of \Cref{def:protoframe} and let $f \colon A \to B$ be a natural transformation such that for each $m \in \bbN$, $f_m \colon A[m] \to B[m]$ induces an equivalence of cofibration categories $\C^{B[m]}_\R \to \C^{A[m]}_\R$. Then $A$ is adequate if and only if $B$ is adequate.
\end{lemma}

\begin{proof}
 It suffices to show that for each simplicial set $K$, the induced functor $\C^{BK}_\R \to \C^{AK}_\R$ is an equivalence of cofibration categories. This can be proven by induction on skeleta with the base case given by the assumption and the inductive steps using the structure of a fibration category on the category of cofibration categories \Cref{thm:fib-cat-of-cof-cat}.
\end{proof}

\begin{proposition}\label{prop:DmxDn-Dmxn}
 For any cofibration category $\C$, the canonical inclusion $D([k] \times [m]) \into D[k] \times D[m]$ induces an equivalence $\C^{D[k] \times D[m]}_\R \to \C^{D([k] \times
 [m])}_\R$ of cofibration categories of diagrams.
\end{proposition}

As a combination of \Cref{lem:DkxD-protoframe,lem:equiv-of-protoframes,prop:DmxDn-Dmxn}, we obtain:

\begin{corollary}\label{cor:Dkx-protoframe}
 For any $k \in \bbN$, the cosimplicial object $D([k] \times [-]) \colon \Delta \to \hoCat$ is adequate.
\end{corollary}

Our next goal is the proof of \Cref{prop:DmxDn-Dmxn}. Our techniques closely follow these of \cite[Sec.\ 23]{dhks} and \cite[Sec.\ 9.5]{radulescu-banu}. In the following series of lemmas, we will assume that $\C$ is a cofibration category and $P$ a finite poset. 

\begin{lemma}\label{lem:max-inv-reedy}
 Let $X \colon
\Sd P \to \C$ be a Reedy cofibrant diagram. Then the restriction
$X|\max^{-1}\{p\}$ is again a Reedy cofibrant diagram.
\end{lemma}

\begin{proof}
 We verify that the inclusion $\max^{-1}\{p\} \into \Sd P$ satisfies the assumptions of \Cref{lem:exactness-criterion}.

 Let $A \in \max^{-1}\{ p \}$, i.e.~$A \subseteq P$ is a chain satisfying $\max A = p$.
We have:
 \begin{align*} \partial(\max{}^{-1}\{ p \} \downarrow A) &= \{B \varsubsetneq A~|~B \neq \varnothing, A
\text{ and } \max B = p\}, \\
  \partial(\Sd P \downarrow A) &= \{B \varsubsetneq A~|~B \neq \varnothing \}. 
\end{align*}
 The map $\partial(\max{}^{-1}\{ p \} \downarrow A) \into \partial(\Sd P \downarrow A)$
factors through:
 \[L := \{ B \subseteq A~|~B\neq \varnothing \text{ and there exists } C \supseteq
B \text{ such that } C \neq A \text{ and } \max C = p\}.\]
 The inclusion $L \into \partial(\Sd P \downarrow A)$ is clearly a sieve. Thus
it remains to show that $\partial(\max{}^{-1}\{ p \} \downarrow A) \into L$ is cofinal. By \cite[Thm.\ IX.3.1]{mac-lane:cwm}, we need to show that for each $B \in L$, the slice category $B \downarrow \partial(\max{}^{-1}\{ p \} \downarrow A)$ is
connected. Explicitly, we have:
 \[ B \downarrow \partial(\max{}^{-1}\{ p \} \downarrow A) = \{ C \supseteq B ~|~ C \neq A \text{ and } \max C = p\}.\]
 This poset has the least element, namely $B \cup \{p\}$, and hence is
connected.
\end{proof}

\begin{lemma}\label{lem:lan-exists-poset}
 Let $X \colon \Sd P \to \C$ be a Reedy cofibrant diagram. Then
the left Kan extension $\Lan_{\max}(X) \colon P \to \C$ exists and is given by $\Lan_{\max}(X)_p = \colim (X|\max{}^{-1}\{ p \})$.
\end{lemma}

\begin{proof}
 For $p \in P$, the obvious inclusion $\max{}^{-1}\{p\} \into (\max \downarrow p)$
 is cofinal and hence, by the pointwise formula for Kan extensions
\cite[Thm.~X.5.1]{mac-lane:cwm}, we have:
\begin{align*}
  \Lan_{\max}(X)_p & = \colim(X|(\max \downarrow p)) \iso \colim (X|\max{}^{-1}\{p\}) \text{.}\qedhere
\end{align*}

\end{proof}

\begin{lemma}\label{lem:equiv-for-Lan}
 Let $A \colon P \to \C$ and $X \colon \Sd P \to \C$ be Reedy cofibrant.
Then a map $\Lan_{\max}(X) \to A$ is a weak equivalence if
and only if its transpose $X \to \max^* A$ is a weak equivalence.
\end{lemma}

\begin{proof}
 We need to show that the following conditions are equivalent:
 \begin{enumerate}
  \item[1.] $\Lan_{\max}(X)_p \to A_p$ is a weak equivalence for all $p \in
P$.
  \item[2.] $X_S \to \max^*A_S$ is a weak equivalence for all $S \in \Sd P$.
 \end{enumerate}

 All morphisms  of the category $\max^{-1}\{p\}$ are weak equivalences and $\{p\}$ is its initial object, so the inclusion $\{p\} \into \max^{-1}\{p\}$ is a homotopy equivalence, and hence, by \Cref{lem:equiv-cofinal}, the induced map $X_{\{p\}} \to \Lan_{\max}(X)_p$ is an equivalence (since $ \Lan_{\max}(X)_p = \colim (X|\max{}^{-1}\{p\})$ by \Cref{lem:lan-exists-poset}). Thus, by 2-out-of-3, 1.\ is equivalent to:
 \begin{enumerate}
 \item[1'.] the composite $X_{\{p\}} \to \Lan_{\max}(X)_p \to A_p$ is a
weak equivalence for all $p \in P$.
 \end{enumerate}
 We will then show that $1'. \Leftrightarrow 2.$.

 For $2. \Rightarrow 1'.$, simply take $S = \{p\}$. For $1'. \Rightarrow 2.$,
consider the following commutative square:
\begin{ctikzpicture}
  \matrix[diagram]
  {
    |(XmS)| X_{\{\max S \}} & |(AmS)| A_{\max S} \\
    |(XS)|  X_S             & |(mAS)| (\max^* A)_S \\
  };

  \draw[->] (XmS) to node[above] {$\we$} (AmS);
  \draw[->] (XS)  to (mAS);
  \draw[->] (XmS) to node[left] {$\we$} (XS);
  \draw[eq] (AmS) to (mAS);
\end{ctikzpicture}
 Since $X$ is homotopical and weak equivalences in $\Sd P$ are created by
$\max$, the vertical left-hand arrow is a weak equivalence. By assumption the
top arrow is a weak equivalence, hence by 2-out-of-3 so is the bottom one.
\end{proof}

\begin{lemma}\label{lem:SdP-P}
 The functor $\max{}^* \colon \C^P \to \C^{\Sd P}$ is a weak equivalence of cofibration categories.
\end{lemma}

\begin{proof}
 Putting $A := \Lan_{\max}(X)$ in \Cref{lem:equiv-for-Lan}, we deduce that
the unit in the diagram 
\begin{ctikzpicture}
  \matrix[diagram,column sep=3.5em,row sep=3.5em]
  {
    |(X)| & |(CP)| \cat{C}^P \\
    |(CSR)| \cat{C}^{\Sd P}_\Reedy & |(CS)| \cat{C}^{\Sd P} \\
  };

  \draw[->,inj] (CSR) to node[below] {$\we$} (CS);
  
  \draw[->] (CP) to node[right] {$\max^*$} (CS);

  \draw[->,dashed] (CSR) to[bend left] node[above left] {$\Lan_{\max}$} (CP);

  \node[rotate=-45,scale=1.3] at ($(X)!0.6!(CS)$) {$\Uparrow$};
\end{ctikzpicture}
is a natural weak
equivalence and hence the composite $\max^* \Lan_{\max}$ is homotopic to a
weak equivalence of \Cref{prop:reedy-vs-levelwise}.(3), thus is itself a weak equivalence.

 So by 2-out-of-3, it suffices to show that $\Lan_{\max}$ is a weak equivalence. We check the Approximation Properties of \Cref{thm:approximation-properties}.

 \textbf{(App1).} Let $X \to Y$ be a map in $\C^{\Sd P}_\R$ whose image
$\Lan_{\max}(X) \to \Lan_{\max}(Y)$ in $\C^P$ is a weak equivalence. We
need to show that $X \to Y$ is a weak equivalence, that is, for all $S \in
\Sd P$, $X_S \to Y_S$ is a weak equivalence. Since both $X$ and $Y$ are
homotopical and weak equivalences in $\Sd P$ are created by $\max$, we have a
commutative diagram:
\begin{ctikzpicture}
  \matrix[diagram]
  {
    |(Xm)| X_{\{\max S\}} & |(Ym)| Y_{\{\max S\}} \\
    |(X)|  X_S            & |(Y)|  Y_S \\
  };

  \draw[->] (Xm) to (Ym);
  \draw[->] (X)  to node[above] {$\we$} (Y);

  \draw[->] (Xm) to node[left]  {$\we$} (X);
  \draw[->] (Ym) to node[right] {$\we$} (Y);
\end{ctikzpicture}
 in which both vertical arrows are weak equivalences. Combining 
\Cref{lem:lan-exists-poset} and the assumption that
for all $p \in P$, $\Lan_{\max}(X)_p \to \Lan_{\max}(Y)_p$ is an equivalence, we
see that the bottom map is a weak equivalence as well. Hence, by 2-out-of-3 so
is the top map.

 \textbf{(App2).} Let $f \colon \Lan_{\max}(X) \to A$. Factor the transpose
$\overline{f} \colon X \to \max^* A$ as a cofibration followed by a weak
equivalence:
\begin{ctikzpicture}
  \matrix[diagram]
  {
    |(X)| X && |(mA)| \max^* A \\
    & |(tA)| \tilde{A} \\
  };

  \draw[->] (X) to node[above] {$\overline{f}$} (mA);

  \draw[->,cof] (X)  to node[below left]  {$i$} (tA);

  \draw[->] (tA) to node[below right] {$w$} node[above left] {$\we$} (mA);
\end{ctikzpicture}
Then we have a commutative square:
\begin{ctikzpicture}
  \matrix[diagram]
  {
    |(LX)| \Lan_{\max}(X)         & |(A0)| A \\
    |(LA)| \Lan_{\max}(\tilde{A}) & |(A1)| A \\
  };

  \draw[->] (LX) to node[left]  {$\Lan_{\max}(i)$} (LA);
  \draw[->] (A0) to node[right] {$\id_A$} (A1);

  \draw[->] (LX) to node[above] {$f$} (A0);

  \draw[->] (LA) to node[above] {$\we$} node[below] {$\overline{w}$} (A1);
\end{ctikzpicture}
 where $\overline{w}$ is the transpose of $w$ and hence, by
\Cref{lem:equiv-for-Lan}, a weak equivalence. Thus (App2) is satisfied.
\end{proof}

\begin{lemma}\label{lem:DmDn-to-Dmn}
 The canonical map $D([k] \times [m]) \into D[k] \times D[m]$ induces an exact functor 
 \[\C^{D[k] \times D[m]}_\R \to \C^{D([k] \times [m])}_\R.\]
\end{lemma}

\begin{proof}
 We check that $D([k] \times [m]) \into D[k] \times D[m]$ satisfies the assumptions of the \Cref{lem:exactness-criterion}. Let $(\varphi, \psi) \colon [l] \to [k] \times [m]$; unpacking the definitions, we see that the latching categories are as follows:
 \begin{align*}\partial\big( D([k] \times [m])  \downarrow (\varphi, \psi)\big) &= \{ A \varsubsetneq [l]~|~ A \neq \varnothing\}, \\
  \partial\big( D[k] \times D[m] \downarrow (\varphi \times \psi)\big) &= \{ A \times B \varsubsetneq [l] \times [l]~|~ A, B \neq \varnothing\},
 \end{align*}
 and the induced map is given by $A \mapsto A \times A$. Let:
 \[ L:= \{ A \times B \subseteq [l] \times [l]~|~ A, B \neq \varnothing \text{ and } A \cup B \neq [l] \}. \]
 The inclusion $L \into \partial\big( D[k] \times D[m] \downarrow (\varphi \times \psi) \big)$ is easily seen to be a sieve; thus, it remains to show that $\partial\big( D([k] \times [m]) \downarrow(\varphi, \psi)\big) \into L$ is cofinal.
Given $A \times B \in L$, the slice category
 \[(A \times B) \downarrow \partial\big( D([k] \times [m]) \downarrow (\varphi, \psi)\big)\]
 is connected since it has the initial object given by $A \times B \into (A \cup B) \times (A \cup B)$ and the result then follows by \cite[Thm.\ IX.3.1]{mac-lane:cwm}.
\end{proof}

\begin{proof}[Proof of \Cref{prop:DmxDn-Dmxn}]
 Consider the following commutative diagram:
 \begin{ctikzpicture}
   \matrix[diagram]
   {
     |(S)| \Sd([k] \times [m]) & |(D)| D([k]\times [m]) & |(DD)| D[k] \times D[m] \\
     & |(i)| [k] \times [m] \\
   };

   \draw[->] (S) to node[above]      {$\circled{1}$} (D);
   \draw[->] (S) to node[below left] {$\circled{2}$} (i);

   \draw[->] (D) to node[above]       {$\circled{3}$} (DD);
   \draw[->] (i) to node[below right] {$\circled{4}$} (DD);

   \draw[->] (D) to (i);
 \end{ctikzpicture}
 By \cite[Lem.\ 3.18]{szumilo:two-models}, \circled{1} induces an equivalence; by \Cref{lem:SdP-P} so does \circled{2}. By \Cref{lem:Dm-to-m}, \circled{3} induces an equivalence, and hence, by 2-out-of-3, so does \circled{4}.
\end{proof}

Let $A \colon \Delta \to \hoCat$ be an adequate cosimplicial object and $\C$ a cofibration category. Define a simplicial set $\Nsub{A}\C$ by:
\[ (\Nsub{A}\C)_m := \left\{\text{homotopical, Reedy cofibrant diagrams } A[m] \to \C \right\}. \]

The reminder of the proof will proceed by introducing a criterion for a map of adequate cosimplicial objects $A \to B$ to induce a categorical equivalence $\Nsub{B}\C \to \Nsub{A}\C$ (\Cref{prop:frames-to-qcat}). We will then deduce the equivalence $\Nf(\C^{D[k]}_\R) \to (\Nf\C)^{\Delta[k]}$ by instantiating this criterion with $D([k] \times [-]) \to D[k] \times D[-]$ (\Cref{thm:compatibility-Dm-Delta-m}).

\begin{proposition}\label{prop:proto-to-qcat}
 For any adequate cosimplicial object $A$ and cofibration category $\C$, $\Nsub{A}\C$ is a quasicategory.
\end{proposition}

\begin{proof}
 By (2), the inclusion $A(\horn{m,i}) \into A[m]$ is a sieve, and hence, by \cite[Lem.\ 1.20]{szumilo:two-models} the induced map $\C^{A[m]}_\R \to \C^{A(\horn{m,i})}_\R$ is a fibration for all $0 < i < m$. By (3), this fibration is acyclic. Thus, by \Cref{lem:reedy-lifting}, there exists a solution to the following lifting problem:
 \begin{ctikzpicture}
   \matrix[diagram]
   {
     |(e)| \varnothing & |(s)| \C^{A[m]}_\R \\
     |(p)| [0]         & |(h)| \C^{\horn{m,i}} \\
   };

   \draw[->] (e) to (s);
   \draw[->] (p) to (h);

   \draw[->,inj] (e) to (p);

   \draw[->,fib] (s) to node[right] {$\we$} (h);
 \end{ctikzpicture}
 This implies that $\Nsub{A}\C$ has fillers for all inner horns.
\end{proof}

\begin{proposition} \label{thm:JNf-CSS}
 Let $A \colon \Delta \to \hoCat$ be an adequate cosimplicial object. Then $\J\Nf(\C^{A[-]}_\R)$ is a complete Segal space.
\end{proposition}

\begin{proof}
 By \Cref{prop:css-is-frame-in-qcat}, it suffices to show that $\J\Nf(\C^{A[-]}_\R)$ is a frame over $\ucat{\J\Nf(\C^{A[-]}_\R)}_0$ in Joyal's model structure. 
 
 We begin by checking that $\J\Nf(\C^{A[-]}_\R)$ is Reedy fibrant, i.e.\ for each $n \in \bbN$, the canonical map $\ucat{\J\Nf(\C^{A[-]}_\R)}_n \to M_n\ucat{\J\Nf(\C^{A[-]}_\R)}$ is an inner isofibration. 
First, let $0<i<m$ and consider the lifting problem:
\begin{ctikzpicture}
  \matrix[diagram]
  {
    |(h)| \horn{m,i} & |(X)|  \ucat{\J\Nf(\C^{A[-]}_\R)}_n  \\
    |(s)| \simp{m}   & |(MX)| M_n\ucat{\J\Nf(\C^{A[-]}_\R)} \\
  };

  \draw[->] (h) to (X);
  \draw[->] (s) to (MX);

  \draw[->,inj] (h) to (s);
  \draw[->,fib] (X) to (MX);
\end{ctikzpicture}
which, by \cite[Lem.\ 1.23]{szumilo:two-models} is equivalent to:
\begin{ctikzpicture}
  \matrix[diagram]
  {
    |(Db)| D \hat{\bdsimp{n}} & |(Cs)| \C^{A[m]}_\R \\
    |(Ds)| D \hat{[n]}        & |(Ch)| \C^{A(\horn{m,i})}_\R \\
  };

   \draw[->] (Db) to (Cs);
   \draw[->] (Ds) to (Ch);

   \draw[->,inj] (Db) to (Ds);

   \draw[->,fib] (Cs) to node[right] {$\we$} (Ch);
\end{ctikzpicture}
The latter admits a solution by \Cref{lem:reedy-lifting}. This implies that the map in question is an inner fibration.

An analogous argument (with condition (4) in place of (3)) shows that the map $\ucat{\J\Nf(\C^{A[-]}_\R)}_n \to M_n\ucat{\J\Nf(\C^{A[-]}_\R)}$ is also an isofibration.

It remains to show that $\J\Nf(\C^{A[-]}_\R)$ is homotopically constant, i.e.\ any simplicial operator $\varphi \colon [n] \to [n']$ induces a categorical equivalence 
$\varphi^* \colon \ucat{\J\Nf(\C^{A[-]}_\R)}_{n'} \to \ucat{\J\Nf(\C^{A[-]}_\R)}_n$. But since all simplicial operators factor as composites of face and degeneracy maps and the latter admit
sections, it suffices to verify it only for inclusions $[n] \into [n']$. We will verify that in this case $\varphi^*$ is in fact an acyclic fibration, i.e.\ every square of the form:
\begin{ctikzpicture}
  \matrix[diagram]
  {
    |(b)| \bdsimp{m} & |(np)| \ucat{\J\Nf(\C^{A[-]}_\R)}_{n'} \\
    |(s)| \simp{m}   & |(n)|  \ucat{\J\Nf(\C^{A[-]}_\R)}_n \\
  };

  \draw[->] (b) to (np);
  \draw[->] (s) to (n);

  \draw[->,inj] (b)  to (s);
  \draw[->,fib] (np) to (n);
\end{ctikzpicture}
admits a diagonal filler. Such a filler corresponds in turn to a lift in
\begin{ctikzpicture}
  \matrix[diagram]
  {
    |(Ab)| A(\bdsimp{m}) & |(np)| \C^{D\ha{[n']}}_\R \\
    |(As)| A[m]          & |(n)|  \C^{D\ha{[n]}}_\R \\
  };

   \draw[->] (Ab) to (np);
   \draw[->] (As) to (n);

   \draw[->,inj] (Ab) to (As);

   \draw[->,fib] (np) to node[right] {$\we$} (n);
\end{ctikzpicture}
which exists, by a similar argument, since $D\ha{[n]} \to D\ha{[n']}$ induces a weak equivalence $\C^{D\ha{[n']}}_\R \to \C^{D\ha{[n]}}_\R$ of cofibration categories.
\end{proof}

\begin{proposition}\label{prop:frames-to-qcat}
 Let $f \colon A \to B$ be a map of adequate cosimplicial objects such that for all $m \in \bbN$, the induced map $f_m^* \colon \C^{B[m]}_\R \to \C^{A[m]}_\R$ is an equivalence of cofibration
categories. Then $f^* \colon \Nsub{B}\C \to \Nsub{A}\C$ is an equivalence of quasicategories.
\end{proposition}

\begin{proof}
 First, notice that for any adequate cosimplicial object $A \colon \Delta \to \hoCat$, the canonical map $\ucat{\J\Nf(\C^{A[-]}_\R)}_0 \to \Nsub{A}\C$, induced by the inclusion $[0] \into D[0]$, is an acyclic fibration. Indeed, the lifting problem:
 \begin{ctikzpicture}
   \matrix[diagram]
   {
     |(b)| \bdsimp{m} & |(S)| \ucat{\J\Nf(\C^{A[-]}_\R)}_0 \\
     |(s)| \simp{m}   & |(Q)| \Nsub{A}\C \\
   };

  \draw[->] (b) to (S);
  \draw[->] (s) to (Q);

  \draw[->,inj] (b) to (s);

  \draw[->] (S) to (Q);
 \end{ctikzpicture}
corresponds to
\begin{ctikzpicture}
  \matrix[diagram]
  {
    |(Ab)| A(\bdsimp{m}) & |(CD)| \C^{D[0]}_\R \\
    |(As)| A[m]          & |(C)|  \C \\
  };

   \draw[->] (Ab) to (CD);
   \draw[->] (As) to (C);

   \draw[->,inj] (Ab) to (As);

   \draw[->,fib] (CD) to node[right] {$\we$} (C);
\end{ctikzpicture}
which has a solution by \Cref{lem:reedy-lifting}, since $[0] \into D[0]$ induces an acyclic fibration of categories of diagrams by \Cref{lem:Dm-to-m} and by (2), $A(\partial \Delta[m]) \into A[m]$ is a sieve.
Thus the vertical maps in the commutative square
\begin{ctikzpicture}
  \matrix[diagram]
  {
    |(SA)| \ucat{\J\Nf(\C^{A[-]}_\R)}_0 & |(SB)| \ucat{\J\Nf(\C^{B[-]}_\R)}_0 \\
    |(QA)| \Nsub{A}\C                   & |(QB)| \Nsub{B}\C \\
  };

  \draw[->] (SA) to (SB);
  \draw[->] (QA) to (QB);

  \draw[->] (SA) to node[left]  {$\we$} (QA);
  \draw[->] (SB) to node[right] {$\we$} (QB);
\end{ctikzpicture}
are categorical equivalences and so, by 2-out-of-3, it suffices to show that so is the top horizontal map. This, however, follows by \Cref{lem:rezk-is-level-qcat} from our 
assumption on $f_m^*$ since $\J\Nf$ carries equivalences of cofibration categories to weak homotopy equivalences of Kan complexes by \Cref{prop:J-preserves-equiv,cor:Nf-preserves-equiv}. 
\end{proof}

\begin{theorem} \label{thm:compatibility-Dm-Delta-m}
 For any cofibration category $\C$ and $k \in \bbN$, the canonical map
\[ \Nf(\C^{D[k]}_\R) \to (\Nf\C)^{\Delta[k]} \]
is a categorical equivalence.
\end{theorem}

\begin{proof}
 Consider adequate cosimplicial objects $A = D([k] \times [-])$ and $B = D[k] \times D[-]$ (see \Cref{cor:Dkx-protoframe,lem:DkxD-protoframe}). By \Cref{prop:frames-to-qcat}, the canonical map $\Nsub{B}\C \to \Nsub{A}\C$ is a categorical equivalence. This, however, completes the proof since $\Nsub{B}\C = \Nf(\C^{D[k]}_\R)$ and $\Nsub{A}\C = (\Nf\C)^{\Delta[k]}$.
\end{proof}

\begin{corollary}
 For any $K \in \sSet$, there is a natural categorical equivalence $\Nf(\C^{DK}_\R) \to (\Nf\C)^K$.
\end{corollary}

\begin{proof}
 Induction on skeleta with the base case given by \Cref{thm:compatibility-Dm-Delta-m}.
\end{proof}

%% file: 4-comp-with-classification.tex
\section{Quasicategory of frames implements simplicial localization}\label{sec:comp-with-classification}

In this section, we prove that the enhancement of the quasicategory of frames of
a cofibration category to a complete Segal space of \Cref{def:bold-Nf}
is equivalent to the classification diagram of Rezk.

\begin{theorem}\label{thm:Nf-and-Ncd}
 For a cofibration category $\C$, the bisimplicial sets $\Ncd\C$ and $\bfNf\C$ are levelwise equivalent and hence Rezk equivalent.
\end{theorem}

The proof of this theorem will be given at the end of the section and throughout we will gather the necessary notions and lemmas.

First off, we are going to need a fattened version of Kan's $\Ex$ functor
which we will denote by $\EX$. For a simplicial set $K$, we define
\begin{align*}
  (\EX K)_n & = \sSet(\N D[n], K) \text{.}
\end{align*}
Notice that by \cite[Lem.\ 3.6]{szumilo:two-models} and the definition of $\EX$, $D \colon \sSet \to \Cat$ is the left adjoint to the composite $\EX \N \colon \Cat \to \sSet$. Moreover, $\EX K$ comes equipped with a map $K \to \EX K$ induced by the functor $p \colon D[n] \to [n]$.

For a cofibration category $\D$, we will consider $\EX\N\w\D$ as an intermediate step in the comparison between $\N\w\D$ and $\J\Nf\D$, which in turn will yield an equivalence between $\N\w(\C^{[m]}_\R)$ and $\J\Nf(\C^{D[m]}_\R)$. Together with \Cref{thm:compatibility-Dm-Delta-m}, this will complete the proof of \Cref{thm:Nf-and-Ncd}.

\begin{lemma}\label{lem:K-to-ExK-equiv}
 For any simplicial set $K$, the map $K \to \EX K$ is a weak homotopy equivalence.
\end{lemma}

This lemma is an instance of \cite[Thm.\ 4.1]{latch-thomason-wilson} with $\theta = D$. For the reader's convenience, we present the specialization of their proof to our case.

\begin{proof}
We begin by noticing that $\EX$ preserves homotopies. Indeed, a homotopy $K \times \Delta[1] \to L$ gives a map 
\[\EX K \times \Delta[1] \to \EX K \times \EX \Delta[1] \to \EX L \]
as desired. Thus, $\EX$ also preserves homotopy equivalences. Similarly, $K^{(-)}$ preserves homotopy equivalences.
 
Now, consider the following commutative square:
\begin{ctikzpicture}
  \matrix[diagram]
  {
    |(i0)| \sSet(\Delta[m] \times \Delta[0], K) & |(in)| \sSet(\Delta[m] \times \Delta[n], K) \\
    |(D0)| \sSet(\N D[m] \times \Delta[0], K)   & |(Dn)| \sSet(\N D[m] \times \Delta[n], K) \\
  };

  \draw[->] (i0) to (in);
  \draw[->] (D0) to (Dn);

  \draw[->] (i0) to (D0);
  \draw[->] (in) to (Dn);
\end{ctikzpicture}
As $m$ and $n$ vary each of the objects becomes a (possibly constant) bisimplicial set.

 First, fix $n \in \bbN$. Then the square becomes:
 \begin{ctikzpicture}
   \matrix[diagram]
   {
     |(i0)| K^{\Delta[0]}      & |(in)| K^{\Delta[n]} \\
     |(E0)| \EX(K^{\Delta[0]}) & |(En)| \EX(K^{\Delta[n]}) \\
   };

   \draw[->] (i0) to node[above] {$\we$} (in);
   \draw[->] (E0) to node[above] {$\we$} (En);

   \draw[->] (i0) to (E0);
   \draw[->] (in) to (En);
 \end{ctikzpicture}
 in which:
 \begin{itemize}
  \item the top map $K^{\Delta[0]} \to K^{\Delta[n]}$ is a homotopy equivalence as the image of the homotopy equivalence $\Delta[n] \to \Delta[0]$ under $K^{(-)}$;
  \item the bottom map $\EX(K^{\Delta[0]}) \to \EX(K^{\Delta[n]})$ is a homotopy equivalence since $\EX$ preserves homotopy equivalences.
 \end{itemize}

 Next, fix $m \in \bbN$. Then the square becomes:
 \begin{ctikzpicture}
   \matrix[diagram]
   {
     |(b0)| \bullet & |(Ki)| K^{\simp{m}} \\
     |(b1)| \bullet & |(KD)| K^{\N D[m]} \\
   };

   \draw[->] (b0) to (b1);

   \draw[->] (Ki) to node[right] {$\we$} (KD);

   \draw[->] (b0) to (Ki);
   \draw[->] (b1) to (KD);
 \end{ctikzpicture}
 and the right hand side vertical map $K^{\Delta[m]} \to K^{\N(D[m])}$ is a homotopy equivalence as the image  under $K^{(-)}$ of $\N p \colon \N D[m] \to \Delta[m]$ which is a homotopy equivalence by \Cref{lem:Dm-to-m}.

 Consequently, applying the diagonal functor $\diag \colon \ssSet \to \sSet$ to this square yields:
 \begin{ctikzpicture}
   \matrix[diagram]
   {
     |(i)| K     & |(b0)| \bullet \\
     |(E)| \EX K & |(b1)| \bullet \\
   };

   \draw[->] (i) to (E);

   \draw[->] (b0) to node[right] {$\we$} (b1);

   \draw[->] (i) to node[above] {$\we$} (b0);
   \draw[->] (E) to node[above] {$\we$} (b1);
 \end{ctikzpicture}
 in which both horizontal and the right vertical map are weak equivalences by the Diagonal Lemma \cite[Thm.\ 4.1.9]{goerss-jardine}. Thus, by 2-out-of-3, $K \to \EX K$ is also a weak equivalence. \qedhere
\end{proof}

For our next argument, we will need an auxiliary lemma about the category of simplicial sets. Our statement is similar to the one proven by Vogt \cite{vogt:help}. Here, we only prove one implication, but under weaker assumptions. 

\begin{lemma}\label{lem:help-vogt}
 Let $f \colon K \to L$ be a map of simplicial sets. Suppose that for each $n \in \bbN$ and a square:
 \begin{ctikzpicture}
   \matrix[diagram]
   {
     |(b)| \bdsimp{n} & |(K)| K \\
     |(s)| \simp{n}   & |(L)| L \\
   };

   \draw[->,inj] (b) to (s);

   \draw[->] (K) to node[right] {$f$} (L);

   \draw[->] (b) to node[above] {$u$} (K);
   \draw[->] (s) to node[above] {$v$} (L);
 \end{ctikzpicture}
there are: a map $w \colon \Delta[n] \to K$ such that $w|\bdsimp{n} = u$ and a homotopy (respectively, an $E[1]$-homotopy) from $fw$ to $v$ relative to the boundary. Then $f$ is a weak homotopy equivalence (respectively, a categorical equivalence).

Moreover, if $L$ is a Kan complex (respectively, a quasicategory), then so is $K$. (Even though $f$ may not be a fibration.)
\end{lemma}

\begin{proof}
 We prove the lemma for weak homotopy equivalences; the proof of categorical equivalences in analogous.

The class of cofibrations $A \to B$ satisfying the lifting property with respect to $f$:
 \begin{ctikzpicture}
   \matrix[diagram]
   {
     |(A)| A & |(K)| K \\
     |(B)| B & |(L)| L \\
   };

   \draw[->,inj] (A) to (B);

   \draw[->] (K) to node[right] {$f$} (L);

   \draw[->] (A) to node[above] {$u$} (K);
   \draw[->] (B) to node[above] {$v$} (L);
 \end{ctikzpicture}
as in the statement of the lemma is closed under (infinite) coproducts, pushouts, and sequential colimits. Thus this lifting property is satisfied by all cofibrations, not only the boundary inclusions. In particular, we can use it for the horn inclusions to see that $K$ is a Kan complex, provided that $L$ is.

Using it with the inclusion $\varnothing \into L$, we obtain a map $g \colon L \to K$ along with a homotopy $H$ from $fg$ to $\id_L$. Consequently, we have a lift in the square:
\begin{ctikzpicture}
  \matrix[diagram]
  {
    |(KK)| K \sqcup K        & |(K)| K \\
    |(Ks)| K \times \simp{1} & |(L)| L \\
  };

  \draw[->,inj] (KK) to (Ks);

  \draw[->] (K) to node[right] {$f$} (L);

  \draw[->] (KK) to node[above] {$[\id_K,gf]$} (K);
  \draw[->] (Ks) to node[below] {$H f$} (L);

  \draw[->,dashed] (Ks) to node[above left] {$G$} (K);
\end{ctikzpicture}
and the commutativity of the upper triangle means that $G$ is a homotopy from $gf$ to $1_K$.
\end{proof}

Next, observe that, for any cofibration category $\D$, the $n$-simplices of the Kan complex $\J\nf\cat{D}$ are the homotopical, Reedy cofibrant diagrams $D\hat{[n]} \to \D$, whereas the $n$-simplices of $\EX \N \w\D$ are all homotopical diagrams $D\hat{[n]} \to \D$. We thus obtain an inclusion $\J\nf\cat{D} \into \EX \N \w\cat{D}$.

\begin{lemma}\label{lem:JNf-to-ExN-equiv}
The inclusion $\J\nf\cat{D} \into \EX \N \w\cat{D}$ is a weak homotopy equivalence.
\end{lemma}

\begin{proof}
It suffices to solve the following lifting problem in the sense of \Cref{lem:help-vogt}:
\begin{ctikzpicture}
  \matrix[diagram]
  {
    |(b)| \bdsimp{n} & |(J)| \J\nf\cat{D} \\
    |(s)| \simp{n}   & |(E)| \EX \N\w\cat{D} \\
  };

  \draw[->] (b) to (s);
  \draw[->] (J) to (E);

  \draw[->] (b) to (J);
  \draw[->] (s) to (E);
\end{ctikzpicture}
A map  $X \colon \Delta[n] \to \EX \N \w{\D}$ corresponds to a homotopical functor $D\hat{[n]} \to \D$ and by commutativity of the square above, the restriction of $X$ to the boundary $\partial \Delta[n]$ is a Reedy cofibrant and homotopical functor:
 \[ D(\hat{\partial \Delta[n]}) \overset{X}{\longrightarrow} \J \Nf \D \subseteq \EX \N \w\D.\]
  Since $D(\hat{\partial \Delta[n]}) \into D[n]$ is a sieve, by \Cref{lem:extend-reedy}, we may find an extension $\widetilde{X}$ and a natural weak equivalence $\widetilde{X} \to X$. Such a natural weak equivalence is a diagram $D\hat{[n]} \times \hat{[1]} \to \w\D$ and the composite $D(\hat{[n]} \times \hat{[1]}) \to D\hat{[n]} \times \hat{[1]} \to \w\D$ gives the desired homotopy by adjunction $D \adjoint \EX \N$.
\end{proof}

\begin{proof}[Proof of \Cref{thm:Nf-and-Ncd}]
First, observe that for every $m$ we have equivalences of cofibration categories
\begin{ctikzpicture}
  \matrix[diagram]
  {
    |(m)|   \cat{C}^{[m]} & |(Dm)| \cat{C}^{D[m]} &
    |(DmR)| \cat{C}^{D[m]}_\Reedy \\
  };
  \draw[->] (m)   to node[above] {$\we$} (Dm);
  \draw[->] (DmR) to node[above] {$\we$} (Dm);
\end{ctikzpicture}
by \Cref{lem:Dm-to-m} and \Cref{prop:reedy-vs-levelwise}, which, by \Cref{prop:J-preserves-equiv,cor:Nf-preserves-equiv}, induce weak homotopy equivalences of simplicial sets
\begin{ctikzpicture}
  \matrix[diagram]
  {
    |(m)|  \J\nf(\cat{C}^{[m]}) & |(Dm)| \J\nf(\cat{C}^{D[m]}) &
    |(DmR)| \J\nf(\cat{C}^{D[m]}_\Reedy) \\
  };
  \draw[->] (m)   to node[above] {$\we$} (Dm);
  \draw[->] (DmR) to node[above] {$\we$} (Dm);
\end{ctikzpicture}

Moreover by \Cref{lem:K-to-ExK-equiv,lem:JNf-to-ExN-equiv}, for any cofibration category $\cat{D}$, we obtain weak homotopy equivalences:
\begin{ctikzpicture}
  \matrix[diagram]
  {
    |(N)| \N\w\cat{D} & |(E)| \EX \N\w\cat{D} & |(J)| \J\nf\cat{D} \\
  };

  \draw[->] (N) to node[above] {$\we$} (E);
  \draw[->] (J) to node[above] {$\we$} (E);
\end{ctikzpicture}

By specializing $\cat{D}$ to $\cat{C}^{[m]}$, $\cat{C}^{D[m]}$
and $\cat{C}^{D[m]}_\Reedy$ we obtain the rows of the diagram
\begin{ctikzpicture}
  \matrix[diagram]
  {
    |(N)|  \N\w(\cat{C}^{[m]})          & |(E)|  \EX \N\w(\cat{C}^{[m]})  & |(J)|  \J\nf(\cat{C}^{[m]}) \\
    |(ND)| \N\w(\cat{C}^{D[m]})         & |(ED)| \EX \N\w(\cat{C}^{D[m]}) & |(JD)| \J\nf(\cat{C}^{D[m]}) \\
    |(NR)| \N\w(\cat{C}^{D[m]}_\Reedy)  & |(ER)| \EX \N\w(\cat{C}^{D[m]}_\Reedy) &
    |(JR)| \J\nf(\cat{C}^{D[m]}_\Reedy) & |(S)| \J(\nf\cat{C})^{\simp{m}} \\
  };

  \draw[->] (N) to node[above] {$\we$} (E);
  \draw[->] (J) to node[above] {$\we$} (E);

  \draw[->] (ND) to node[above] {$\we$} (ED);
  \draw[->] (JD) to node[above] {$\we$} (ED);

  \draw[->] (NR) to node[below] {$\we$} (ER);
  \draw[->] (JR) to node[below] {$\we$} (ER);

  \draw[->] (JR) to node[below] {$\we$} (S);

  \draw[->] (N)  to (ND);
  \draw[->] (NR) to (ND);

  \draw[->] (E)  to (ED);
  \draw[->] (ER) to (ED);

  \draw[->] (J)  to node[right] {$\we$} (JD);
  \draw[->] (JR) to node[right] {$\we$} (JD);
\end{ctikzpicture}
where the bottom right map is a weak homotopy equivalence by \Cref{thm:compatibility-Dm-Delta-m,prop:J-preserves-equiv};
and so are the maps of the right column by the preceding discussion.
Therefore, all the maps in the diagram are weak homotopy equivalences.

The shortest zig-zag of weak homotopy equivalences connecting $\N(\w\cat{C}^{[m]})$
to $\J(\nf\cat{C})^{\simp{m}}$ that we can extract is
\begin{ctikzpicture}
  \matrix[diagram]
  {
    |(N)| \Ncd\C =  \N\w(\cat{C}^{[m]})          & |(ED)| \EX \N\w(\cat{C}^{D[m]}) &
    |(JR)| \J\nf(\cat{C}^{D[m]}_\Reedy) & |(S)| \J(\nf\cat{C})^{\simp{m}} = \bfNf\C \text{.} \qedhere \\
  };

  \draw[->] (N)  to node[above] {$\we$} (ED);
  \draw[->] (JR) to node[above] {$\we$} (ED);
  \draw[->] (JR) to node[above] {$\we$} (S);
\end{ctikzpicture} 
\end{proof}

The categories with weak equivalences $\sSet$ and $\ssSet$ admit model structures, known as Joyal's \cite[Thm.\ 6.12]{joyal:theory-of-quasi-cats} and Rezk's \cite[Thm.\ 7.2]{rezk:css} model structures, respectively. The functor $\ev_0 \colon \ssSet \to \sSet$ defined by $\ev_0(W) = \ucat{W}_0$  is a right Quillen functor and a Quillen equivalence \cite[Thm.\ 4.11]{joyal-tierney:qcat-vs-segal}. It follows that its right derived functor $\bbR\ev_0 \colon \ssSet \to \sSet$ exists and is a DK-equivalence.

\begin{corollary}\label{cor:Nf-and-evNcd}
 For any cofibration category $\C$, the quasicategories $\Nf\C$ and $(\bbR \ev_0) \Ncd \C$ are equivalent.
\end{corollary}

\begin{remark}
By \cite[Thm.\ 6.3]{toen:unicity}, the simplicial set of derived autoequivalences of $\ssSet$ is equivalent to $\bbZ/2$, which therefore acts freely and transitively on the set of homotopy classes of derived equivalences $\weCat \to \ssSet$. Hence there are two homotopy classes, represented by $\Ncd$ and $\Ncd^\op$, respectively. One recognizes the class of such $F$ by the following criterion: the diagram of solid arrows
\begin{ctikzpicture}
  \matrix[diagram]
  {
    |(F0)| F[0] & |(N0)| \Ncd{}[0] \\
    |(F1)| F[1] & |(N1)| \Ncd{}[1] \\
  };

  \draw[->] (F0) to node[left] {$F\delta_0$} (F1);

  \draw[->,dashed] (N0) to (N1);

  \draw[->] (F0) to node[above] {$\we$} (N0);
  \draw[->] (F1) to node[above] {$\we$} (N1);
\end{ctikzpicture}
 can be completed to a (homotopy) commutative square in two ways, either with $\Ncd\delta_0$ or $\Ncd\delta_1$. The former implies $F \sim \Ncd$ and the latter $F \sim \Ncd^\op$. It follows by \Cref{thm:Nf-and-Ncd} that the restriction of such $F$ to the category of cofibration categories is equivalent to either $\bfNf$ or $\bfNf^\op$.
 
 Since $\bbR\ev_0 \colon \ssSet \to \sSet$ is an equivalence, there are two homotopy classes of derived equivalences $\weCat \to \sSet$ represented by the composites: $(\bbR\ev_0) \Ncd$ and $(\bbR\ev_0) \Ncd^\op$. Thus the restriction of such an equivalence to the category of cofibration categories is equivalent to either $\Nf$ or $\Nf^\op$. One example of such an equivalence (equivalent to $\Nf$) is the composite of the hammock localization of Dwyer and Kan \cite{dwyer-kan:simplicial-localizations, dwyer-kan:calculate-localizations} followed by the derived homotopy coherent nerve \cite{cordier:homotopy-coherent} (these are indeed equivalences by \cite[Thm.\ 1.7]{barwick-kan:characterization} and e.g.\ \cite[Sec.\ 1.5]{lurie:htt} or \cite[Cor.\ 8.2]{dugger-spivak:mapping-spaces}, respectively).
\end{remark}

%% file: 5-comp-in-model.tex
\section{Frames in model categories} \label{sec:comp-in-model}

Let $\cat{M}$ be a model category.
Then its full subcategory of cofibrant objects $\cat{M}_\cof$ inherits
a structure of a cofibration category.
Dually, the full subcategory of fibrant objects $\cat{M}_\fibr$
is a fibration category.
Thus there are two different quasicategories of frames associated to $\cat{M}$:
$\nf \cat{M}_\cof$ and $\nf \cat{M}_\fibr$ (these two $\Nf$'s are, of course, different functors).
It follows from \Cref{cor:Nf-and-evNcd} and its dual that
these two quasicategories are naturally equivalent.
However, the resulting zig-zag of equivalences is rather long and unwieldy.
In this section, we discuss an alternative and much more direct comparison
involving only a single fraction.

To this end we introduce an enhanced version of the quasicategory of frames
that utilizes both the cofibrations and the fibrations of $\cat{M}$. For this reason we need to use Reedy categories as opposed to direct categories. Recall that a \textbf{Reedy category} is a category $I$, equipped with two wide subcategories $I_\sharp$ and $I_\flat$ (whose morphisms are called the \textbf{face operators} and \textbf{degeneracy operators}, respectively) such that:
\begin{enumerate}
 \item there exists a function $\deg \colon \ob I \to \bbN$ making $I_\sharp$ into a direct category and $I_\flat$ into an inverse category (i.e.\ opposite of a direct category);
 \item every morphism of $I$ factors uniquely as the composite of a degeneracy operator followed by a face operator.
\end{enumerate}

For a small category $J$, define a homotopical category $\mD J$ as follows.
Objects of $\mD J$ are all functors $[s] \times [t] \to J$
for varying $s$ and $t$.
A morphism from $x \from [s] \times [t] \to J$
to $x' \from [s'] \times [t'] \to J$ is a pair of face operators
$\phi \from [s] \ito [s']$ and $\psi \from [t'] \ito [t]$
\st{} $x (\id, \psi) = x' (\phi, \id)$ (as functors $[s] \times [t'] \to J$).
There is a functor $\mD J \to J$ that evaluates $x \from [s] \times [t] \to J$
at $(s, 0)$ and weak equivalences of $\mD J$ are created by this functor
(from the isomorphisms of $J$).
The category $\mD J$ is a Reedy category where a morphism $(\phi, \psi)$
as above is a face operator if $\psi = \id$ and a degeneracy operator
if $\phi = \id$.
The unique factorization of $(\phi, \psi)$ as the composite of
a degeneracy operator and a face operator is
$(\phi, \psi) = (\phi, \id) (\id, \psi)$.

For a model category $\cat{M}$, define a simplicial set $\mnf \cat{M}$ by:
 \[(\mnf \cat{M})_m :=\left\{ \text{homotopical, Reedy cofibrant and fibrant diagrams } \mD[m] \to \M \right\} . \]
This is indeed a simplicial set since every simplicial operator $\varphi \colon [m] \to [n]$ induces isomorphisms of all latching and matching categories of $\mD[m]$ and $\mD[n]$, and thus preserves Reedy (co)fibrancy.
We will prove that it is a quasicategory naturally equivalent to both
$\nf \cat{M}_\cof$ and $\nf \cat{M}_\fibr$.

For a category $J$, we introduce the following functors
relating $DJ$ and $\mD J$:

\begin{align*}
  i \from DJ \to \mD J\text{,}   & & x \from [s] \to J & & \mapsto & & x \from [s] \times [0] \to J \\
  q \from \mD J \to DJ\text{,}   & & x \from [s] \times [t] \to J & & \mapsto & & x \from [s] \times \{0\} \to J \\
  s \from \mD J \to \mD J\text{,}& & x \from [s] \times [t] \to J & & \mapsto & & x \sigma_0 \from [s] \times [t+1] \to J\\
\end{align*}
Then we have $qi = \id_{DJ}$, $si = i$ and there are natural weak equivalences
$\kappa \from s \to iq$ and $\lambda \from s \to \id_{\mD J}$.
All components of $\kappa$ are degeneracy operators of $\mD J$ that are dual to
the inclusions $[0] \ito [t+1]$.
Similarly, components of $\lambda$ are dual to the face operators
$\delta_0 \from [t] \ito [t+1]$.
It follows that both $\kappa i$ and $\lambda i$ are equal to $\id_i$.

The definition of $\mD$ could be extended to general simplicial sets,
but we will only use one such ad hoc extension.
Namely, we define $\mD \bdsimp{m}$ as the full subcategory of $\mD [m]$
spanned by all non-surjective functors $[s] \times [t] \to [m]$.
All the functors and transformations introduced above are natural in $J$ as well
as with respect to the inclusions $\bdsimp{m} \ito \simp{m}$.
Denote the induced inclusions $u \from D\bdsimp{m} \ito D[m]$
and $\bar{u} \from \mD\bdsimp{m} \ito \mD[m]$.

\begin{theorem}\label{thm:Nf-model-cats}
  For a model category $\cat{M}$, the simplicial set $\mnf \cat{M}$
  is a quasicategory.
  Moreover, the functors $i \from D[m] \to \mD[m]$ induce an equivalence
  $\mnf \cat{M} \to \nf \cat{M}_\cof$.
\end{theorem}

All the constructions above, as well as the theorem, readily dualize to yield
an equivalence $\mnf \cat{M} \to \nf \cat{M}_\fibr$.

A map of Reedy categories $I \to J$ is a \textbf{bisieve} if
it carries face operators to face operators
and the induced functor $I_\sharp \to J_\sharp$ is a sieve,
and, dually, it carries degeneracy operators to degeneracy operators
and the induced functor $I_\flat \to J_\flat$ is a cosieve.

\begin{lemma}
\label{bisieve-lift}
  Let $J$ be a Reedy category and $I \ito J$ a bisieve.
  Let $X \to Y$ be a morphism of $J$-diagrams in a model category $\cat{M}$
  with $X$ Reedy cofibrant.
  Then any factorization
  \begin{ctikzpicture}
    \matrix[diagram]
    {
      |(X)| X|I & |(tX)| \tilde{X}_I & |(Y)| Y|I \\
    };

    \draw[->]  (X) to node[above] {$\we$} (tX);
    \draw[fib] (tX) to (Y);
  \end{ctikzpicture}
  into a weak equivalence (not necessarily a cofibration) and a Reedy fibration
  \st{} $\tilde{X}_I$ is Reedy cofibrant lifts to a factorization
  \begin{ctikzpicture}
    \matrix[diagram]
    {
      |(X)| X & |(tX)| \tilde{X} & |(Y)| Y \\
    };

    \draw[->]  (X) to node[above] {$\we$} (tX);
    \draw[fib] (tX) to (Y);
  \end{ctikzpicture}
  into a weak equivalence and a Reedy fibration
  \st{} $\tilde{X}$ is Reedy cofibrant.
\end{lemma}

\begin{proof}
  The argument is essentially the same as the standard construction of
  Reedy factorizations (see e.g.\ \cite[Lem.\ 7.4]{riehl-verity:theory-and-practice}).
  By induction, it suffices to extend the given factorization over
  an object $j \in J$ of a minimal degree among these not in $I$.
  Given such and object consider the following diagram.
  \begin{ctikzpicture}
    \matrix[diagram]
    {
      |(LX)| L_j X & |(LtX)| L_j \tilde{X} & & |(LY)| L_j Y \\
      |(X)|  X_j   & |(po)| \bullet        & |(tX)| \tilde{X}_j & |(pb)| \bullet & |(Y)| Y \\
      & |(MX)| M_j X & & |(MtX)| M_j \tilde{X} & |(MY)| M_j Y \\
    };

    \draw[->] (LX) to node[above] {$\we$} (LtX);
    \draw[->] (X)  to node[below] {$\we$} (po);

    \draw[cof] (LX) to (X);
    \draw[cof] (LtX) to (po);

    \draw[->] (MtX) to (MY);
    \draw[->] (pb) to (Y);

    \draw[->] (pb) to (MtX);
    \draw[->] (Y) to (MY);

    \draw[->] (LtX) to (LY);
    \draw[->] (LY) to[bend left] (Y);

    \draw[->] (X) to[bend right] (MX);
    \draw[->] (MX) to (MtX);

    \draw[cof,dashed] (po) to node[below] {$\we$} (tX);
    \draw[fib,dashed] (tX) to (pb);

  \end{ctikzpicture}
  Here, $L_j$ and $M_j$ denote the latching and matching objects at $j$.
  The morphism $L_j X \to X_j$ is a cofibration since $X$ is Reedy cofibrant
  and $L_j X \to L_j \tilde{X}_j$ is a weak equivalence since $X \to \tilde{X}$
  is a weak equivalence of Reedy cofibrant objects.
  The two objects denoted by bullets are formed by taking the pushout
  on the left and the pullback on the right and $\tilde{X}_j$ arises from
  a factorization of the resulting morphism.

  This extends the original factorization over the subcategory $I'$,
  i.e.\ the bisieve generated by $I$ and $j$.
  Denote the resulting diagram $\tilde{X}_{I'}$.
  The composite $L_j \tilde{X} \to \bullet \to \tilde{X}_j$ is a cofibration
  so $\tilde{X}_{I'}$ is Reedy cofibrant.
  The composite $X_j \to \bullet \to \tilde{X}_j$ is a weak equivalence
  and hence so is the morphism $X|I' \to \tilde{X}_{I'}$.
  Finally, the map $\tilde{X}_j \to \bullet$ is a fibration and thus
  $\tilde{X}_{I'} \to Y|I'$ is a Reedy fibration.
\end{proof}

\begin{proof}[Proof of \Cref{thm:Nf-model-cats}]
  First, observe that $i^*$ is indeed a simplicial map since each
  $i \from D[m] \ito \mD[m]$ induces isomorphisms of latching categories
  so that $i^*$ preserves Reedy cofibrant diagrams.

  By \Cref{lem:help-vogt}, it suffices to consider a square
  \begin{ctikzpicture}
    \matrix[diagram]
    {
      |(b)| \bdsimp{m} & |(mn)| \mnf \cat{M} \\
      |(s)| \simp{m}   & |(n)|  \nf \cat{M}_\cof \\
    };

    \draw[->] (b)  to (s);
    \draw[->] (mn) to (n);

    \draw[->] (b)  to node[above] {$X$} (mn);
    \draw[->] (s)  to node[below] {$Y$} (n);
  \end{ctikzpicture}
  and find a map $Z \from \simp{m} \to \mnf \cat{M}$ that makes
  the upper triangle commute and the lower one commute up to $E[1]$-homotopy
  relative to $\bdsimp{m}$.
  In particular, it then follows
  that $\mnf \cat{M}$ is a quasicategory since $\nf \cat{M}_\cof$ is.

  We have a Reedy fibrant and cofibrant diagram
  $X \from \mD\bdsimp{m} \to \cat{M}$
  and a Reedy cofibrant diagram $Y \from D[m] \to \cat{M}$ \st{} $Y u = X i$.
  Therefore, we have $Y q \bar{u} = Y u q = X i q$.
  We will correct $Y q$ to a Reedy fibrant and cofibrant diagram $Z$ so that
  $Z \bar{u} = X$ and there is a weak equivalence $Z i \weto Y$
  relative to $D\bdsimp{m}$.

  First, observe that $\kappa$ and $\lambda$ yield natural weak equivalences
  \begin{ctikzpicture}
    \matrix[diagram]
    {
      |(ip)| X i q & |(s)| X s & |(id)| X \\
    };

    \draw[->] (s) to node[above] {$\we$} (ip);
    \draw[->] (s) to node[above] {$\we$} (id);
  \end{ctikzpicture}
  relative to $D\bdsimp{m}$.
  Factor the resulting morphism $X s \to X \times X i q$ into a weak equivalence
  and a Reedy fibration
  \begin{ctikzpicture}
    \matrix[diagram]
    {
      |(s)| X s & & |(i)| X \times X i q \\
      & |(t)| \tilde{X} \\
    };

    \draw[->] (s) to (i);

    \draw[->] (s) to node[below left] {$w$} node[above right] {$\we$} (t);

    \draw[fib] (t) to node[below right] {$(r,r')$} (i);
  \end{ctikzpicture}
 so that the restriction to $D\bdsimp{m}$ is a path object factorization; in particular, the restriction of $w$ to $D \bdsimp{m}$ is a section of the restrictions of both $r$ and $r'$. Here, $r$ and $r'$ are weak equivalences and $r'$ is also a Reedy fibration
  (since $X$ is Reedy fibrant).
  Hence $r'$ admits a section $t$ since $X i q$ is Reedy cofibrant
  ($q$ induces isomorphisms of latching categories).
  Moreover, $t$ can be chosen to agree with $w$ on $D\bdsimp{m}$
  since $u$ is a bisieve.
  Thus, the composite $r t$ is a weak equivalence $X i p \weto X$
  relative to $D\bdsimp{m}$, i.e.\ $X$ is a Reedy fibrant replacement
  of $X i q$ relative to $D\bdsimp{m}$.
  Since $Y q \bar{u} = X i q$, we can lift it to
  a Reedy fibrant replacement $Y q \weto Z$ relative to $D\bdsimp{m}$
  with $Z$ Reedy cofibrant using \Cref{bisieve-lift}.

  Then we have $Z \bar{u} = X$ so that $Z$ makes the upper triangle commute.
  Moreover, the induced weak equivalence $Y = Y q i \weto Z i$
  is relative to $D\bdsimp{m}$ and hence induces an $E[1]$-homotopy relative to
  $\bdsimp{m}$ in the lower square by \cite[Lem.\ 4.6]{szumilo:two-models}.
\end{proof}